\documentclass[11pt]{amsart}
\usepackage[top=1in, bottom=1in, left=1.2in, right=1.2in]{geometry}
\usepackage{amsmath}
\usepackage{amssymb}
\usepackage{amsthm}
\usepackage{verbatim}
\usepackage{graphicx}
\usepackage{graphics}
\usepackage{enumerate}
\usepackage{mathrsfs}
\usepackage{booktabs}
\usepackage[toc,page]{appendix}
\usepackage{enumerate}
\usepackage{mathrsfs}
\usepackage{fancyhdr}
\usepackage{latexsym, amsfonts, amssymb, amsmath,  amsthm}
\usepackage{amsopn, amstext, amscd,pifont}
\usepackage{amssymb,bm, cite}
\usepackage[all]{xy}
\usepackage{color}
\usepackage{blkarray}
\usepackage{graphicx,hyperref}
\usepackage{float}
\usepackage{dsfont}
\usepackage{tikz-cd}

\theoremstyle{plain}

\newtheorem{theorem}{Theorem}[section]
\newtheorem{proposition}[theorem]{Proposition}
\newtheorem{lemma}[theorem]{Lemma}
\newtheorem{corollary}[theorem]{Corollary}

\theoremstyle{definition}

\newtheorem{example}[theorem]{Example}

\def\Z{\mathbb{Z}}

\def\Z2{\mathbb{Z}_{2}}
\def\m2#1{\ ({\rm mod} \ 2^{#1})}

\numberwithin{equation}{section}

\begin{document}

\title{Quantitative equidistribution of angles of multipliers}

\author{Yan Mary He}
\address{Department of Mathematics, University of Toronto, M5S 2E4, Canada}
\email{yanmary.he@mail.utoronto.ca}

\author{Hongming Nie}
\address{Facultad de Matem\'aticas, Pontificia Universidad Cat\'olica de Chile, Santiago, Chile}
\email{hongming.i.nie@gmail.com}

\begin{abstract}
We study angles of multipliers of repelling cycles for hyperbolic rational maps in $\mathbb{C}(z)$. For a fixed $K \gg 1$, we show that almost all intervals of length $2\pi/K$ in $(-\pi,\pi]$ contain a multiplier angle with the property that the norm of the multiplier is bounded above by a polynomial in $K$. 
%
%and study the small scale statistics of their angles. For a fixed hyperbolic rational map $h$ and $K\gg1$, we state an upper bound of the absolute values of the multipliers for $h$ such than almost all intervals of length $2\pi/K$ in $[0,2\pi)$ contains at least one angle of these multipliers. Moreover, we also bound the variance of the numbers of such angles in a short interval. 
\end{abstract}
%\subjclass[2010]{Primary 37P05; Secondary 11S82, 37B05}
%\keywords{$p$-adic dynamics, Julia set, geometrically finite, countable Markov chain}

\maketitle

\section{Introduction}\label{sec:intro}
Let $h \in\mathbb{C}(z)$ be a hyperbolic rational map of degree at least $2$. We consider the dynamical system $h : J(h) \to J(h)$ where $J(h)$ is the Julia set of $h$. We denote by $\delta$ the Hausdorff dimension of the Julia set $J(h)$ and by $h^k$ the $k$-th iterate of $h$ for $k\ge 1$. A periodic orbit $\hat{z} = \{z, h(z),...,h^{n-1}(z) \}$ of period $n \ge 1$ is called {\it primitive} if $h^n(z) = z$ and $h^m(z) \neq z$ for any $1 \le m <n$. Let $\mathcal{P}$ be the set of primitive periodic orbits of $h$ in $J(h)$. If $\hat{z} \in \mathcal{P}$ has period $n$, the quantity
$$\lambda(\hat{z}) := (h^n)'(z)$$ (calculated in local coordinates) is called the \emph{multiplier} of $\hat{z}$. Its {\it holonomy} is given by $$\lambda_\theta(\hat{z}) := \frac{\lambda(\hat{z})}{|\lambda(\hat{z})|}.$$

A celebrated result of Oh and Winter \cite{Oh17} states that if $h$ is not conjugate to a monomial $z\mapsto z^{\pm{\deg(h)}}$, then there exists $\eta > 0$ such that \begin{equation}\label{equ:number}
\mathcal{N}_t:=\#\{\hat z\in\mathcal{P}:|\lambda(\hat z)|<t\}={\rm Li}(t^\delta)+O(t^{\delta-\eta})
\end{equation}
where ${\rm Li}(x) = \int_{2}^{x} \frac{du}{\log u} \sim x/\log x$ as $x \to \infty$. We write $g_1(x) = O(g_2(x))$ as $x \to \infty$ if there exists $C>0$ and $x_0 \in \mathbb R$ such that $|g_1(x)| \le C|g_2(x)|$ for all $x \ge x_0$. We also write $g_1(x) \sim g_2(x)$ as $x \to \infty$ if $\lim_{x \to \infty} g_1(x)/g_2(x) = 1$. 

Moreover, if the Julia set $J(h)$ is not contained in a circle in $\widehat{\mathbb{C}}$, Oh and Winter showed that for any $\psi\in C^4(\mathbb{S}^1)$, we have
\begin{equation}\label{equ:epsilon}
\sum_{\hat z\in\mathcal{P}:|\lambda(\hat z)|<t}\psi(\lambda_\theta(\hat z))=\int_0^1\psi(e^{2\pi i\theta})d(\theta)\cdot\mathrm{Li}(t^\delta)+O(t^{\delta-\eta}).
\end{equation}
In particular, for a fixed interval $I \subset (-\pi, \pi]$, we have
\begin{equation} \label{equ:introI}
\frac{\# \{ \hat{z} : |\lambda(\hat{z})| < t, {\rm Arg}(\lambda(\hat{z})) \in I \}}{\mathcal{N}_t} \sim \frac{|I|}{2\pi}, \text{ as } t \to \infty
\end{equation}
where $|I|$ is the length of the interval $I$. The notation ${\rm Arg}(z)$ denotes the principal argument of a complex number $z$. 

In this paper we study the existence of multiplier angles in small intervals and obtain a quantitative equidistribution result on angles of multipliers. 
%It is not hard to see that (\ref{equ:introI}) does not hold for intervals $I$ with $|I| < 2\pi \mathcal{N}_t^{-1}$ for a fixed $t \gg 1$ 
For a fixed $t \gg 1$, if we divide $(-\pi, \pi]$ equally into more than $\mathcal{N}_t$ intervals, then
it is not hard to see that there exists an interval $I$ with $|I| < 2\pi \mathcal{N}_t^{-1}$ for which (\ref{equ:introI}) does not hold. Therefore, for small intervals, one can expect at most a statistical statement about the existence of multiplier angles. We show in this paper that for a fixed $K \gg 1$, almost all intervals of length $2\pi/K$ in $(-\pi, \pi]$ contain a multiplier angle with the property that the norm of the multiplier is bounded above by a polynomial in $K$. Our result is inspired by recent results in analytic number theory, in particular by Rudnick-Waxman \cite{Rudnick19} and Parzanchevski-Sarnak \cite{Parzanchevski18}. Another equidistribution result on multipliers of hyperbolic rational maps was recently obtained by Sharp-Stylianou \cite{Sharp20}.

\subsection{Statement of main result}
We now describe our results in more details. Given $K \gg 1$, we divide the interval $(-\pi, \pi]$ into $K$ disjoint intervals of equal length and study the number of multiplier angles ${\rm Arg}(\lambda(\hat{z}))$ falling into each such interval subject to the constraint $|\lambda(\hat{z})|<t$ for some fixed $t \gg 1$. If the number $K$ of intervals is larger than the number $\mathcal{N}_t$ of angles, then there exists an interval which does not contain any multiplier angles. Our main theorem shows that if $\mathcal{N}_t$ (or equivalently $t$) is greater than a certain power of $K$, then almost all intervals of size $2\pi/K$ contain at least one multiplier angle ${\rm Arg}(\lambda(\hat{z}))$ with $|\lambda(\hat{z})|<t$. 

\begin{theorem}\label{thm:main}
	Let $h\in\mathbb{C}(z)$ be a hyperbolic rational map of degree at least $2$. Suppose that $J(h)$ is not contained in a circle in $\widehat{\mathbb{C}}$.
	Then for almost all $\theta \in (-\pi, \pi]$, the interval $\left[\theta -\frac{\pi}{K},\theta + \frac{\pi}{K} \right]$ contains at least one multiplier angle ${\rm Arg}(\lambda(\hat{z}))$ with 
	$$|\lambda(\hat{z})| \le K^{\frac{7}{2\alpha}}$$
	where $\alpha= \min \left\{\frac{\delta}{2}, 2\eta \right\}$, and $\delta$ and $\eta$ are as in (\ref{equ:number}).
	%Moreover, if the conjecture (\ref{equ:conjecture}) holds, the above bound of $t$ is sharp.
\end{theorem}

Here by almost all, we mean that the probability of $\theta$ not satisfying the desired property has an upper bound which tends to zero as $K$ tends to infinity.
Our result suggests a polynomial bound for $|\lambda(\hat{z})|$ in terms of $K$. If the angles of multipliers are equidistributed for sufficiently large $t$, we expect that a sharp bound is given by
$$t=K^{\frac{1}{\delta}}(\log K)^{\frac{1}{\delta}+o(1)}.$$ 

\subsection{Strategy of the proofs} \label{sec:strategy}
We prove the theorem by studying the variance of suitable smooth counting functions. More specifically, for $\theta \in (-\pi, \pi]$, let 
$$\mathcal{N}_{K,t}(\theta) :=\#\{\hat z\in\mathcal{P}:|\lambda(\hat z)|<t, \mathrm{Arg}(\lambda(\hat z))\in \left[\theta-\frac{\pi}{K}, \theta+\frac{\pi}{K}\right]\}$$
be the number of primitive periodic orbits with multipliers having norm smaller than $t$ and angles in the interval of length $2\pi/K$ centered at $\theta$.

Recall that the expected value of $\mathcal{N}_{K,t}(\theta)$ is 
$$\mathbb{E}(\mathcal{N}_{K,t}) = \int_{-\pi}^{\pi}\mathcal{N}_{K,t}(\theta)\frac{d\theta}{2\pi}$$
and the variance of $\mathcal{N}_{K,t}(\theta)$ is given by
$${\rm Var}(\mathcal{N}_{K,t}) = \int_{-\pi}^{\pi} |\mathcal{N}_{K,t} - \mathbb{E}(\mathcal{N}_{K,t})|^2\frac{d\theta}{2\pi}.$$

If $\mathcal{N}_t \ll K$, we prove in Section \ref{sec:trivial} that the variance of $\mathcal{N}_{K,t}$ is
$${\rm Var} (\mathcal{N}_{K,t}) \sim \frac{\mathcal{N}_t}{K}.$$ 

The more interesting case is when $K\ll\mathcal{N}_t$. In this case, it is not straightforward to obtain an effective bound for the variance ${\rm Var}(\mathcal{N}_{K,t})$. Alternatively, adopting ideas from analytic number theory, we consider a smoothed version $\Phi^*_{K,t}$ of the counting function $\mathcal{N}_{K,t}$ as follows:
\begin{equation} \label{eq:introbarPhi}
\Phi^*_{K,t}(\theta) :=\sum_{n \ge1} \frac{1}{n} \sum_{\substack{h^n(z) = z\\ \hat z\in\mathcal{P}}} \phi \left(\frac{|(h^n)'(z)|}{t}\right)F_K\left( {\rm Arg}((h^n)'(z))-\theta\right)\log|(h^n)'(z)|
\end{equation}
where $\phi$ and $F_K$ are some cut-off functions for the norm $|\lambda(\hat{z})|$ and the angle ${\rm Arg}(\lambda(\hat{z}))$, respectively. The function $\Phi^*_{K,t}(\theta)$ is a smooth count for the number of primitive periodic orbits with multipliers whose norms are less than $t$ and angles lie in $[\theta-\pi/K, \theta+\pi/K]$. 

The main challenge of this paper is to derive technical estimates on various quantities which lead to a bound on the variance of $\Phi^*_{K,t}$. A key quantity we consider is the {\it non-primitive} version $\Phi_{K,t}(\theta)$ of the smooth function $\Phi^*_{K,t}(\theta)$. The function $\Phi_{K,t}$ is derived from $\Phi^*_{K,t}$ by taking the same formula as (\ref{eq:introbarPhi}) except that we sum over periodic points $z$ satisfying $h^n(z) = z$ that are not necessarily primitive. An important step is to find an upper bound for the variance of $\Phi_{K,t}$. To achieve this, we use methods from analytic number theory; namely, we apply the Mellin inversion formula to the function $\phi$ so that we can write $\Phi_{K,t}(\theta)$ as an expression involving the logarithmic derivative of the dynamical $L$-function for the dynamical system $(J(h),h)$. The work of Oh-Winter \cite{Oh17} provides an upper bound for this derivative, and therefore with some additional work we are able to obtain an upper bound for the variance of $\Phi_{K,t}$. Then to obtain a bound on the variance of $\Phi^*_{K,t}$, we consider the difference $\Phi_{K,t} - \Phi^*_{K,t}$ between the non-primitive and the primitive functions. In particular, we bound its $L^2$-norm and its expectation. Finally, we prove Theorem \ref{thm:51}, which is a slightly stronger version of Theorem \ref{thm:main}, by using these estimates together with the Chebyshev inequality.

\subsection{Organization of the paper}
The rest of the paper is organized as follows. In Section \ref{sec:prelim} we review basic definitions and results in complex dynamics and analytic number theory. We define the functions $\Phi_{K,t}(\theta)$ and $\Phi^*_{K,t}(\theta)$ as mentioned above in Section \ref{sec:defs}. Section \ref{sec:var} is the main technical section where we obtain bounds on the expected values and the variances of $\Phi_{K,t}$, $\Phi^*_{K,t}$ and their difference. Then we prove Theorem \ref{thm:main} in Section \ref{sec:proofMainThm}. Finally, we study the variance of $\mathcal{N}_{K,t}$ when $K\gg\mathcal{N}_t$ in Section \ref{sec:trivial}.

\section{Preliminaries} \label{sec:prelim}
In this section, we review some basic definitions and recent results regarding multipliers of hyperbolic rational maps and the dynamical $L$-function for dynamics on hyperbolic Julia sets. In the end we discuss the Mellin transform from analytic number theory and the construction of the test function $\phi$ in formula \eqref{eq:introbarPhi}. Standard references are \cite{Iwaniec04}, \cite{Milnor06} and \cite{Oh17}.

\subsection{Hyperbolic rational maps and multipliers}
Let $h\in\mathbb{C}(z)$ be a rational map of degree at least $2$. We say that $h$ is \emph{hyperbolic} if there exist $C>1$ and a smooth conformal metric $\rho$ defined on a neighborhood of its Julia set $J(h)$ such that $||h'(z)||_\rho>C$ for all $z\in J(h)$. For equivalent definitions, we refer the reader to \cite{McMullen94}. 

Consider the dynamical system $h : J(h) \to J(h)$. Recall that $\mathcal{P}$ is the set of primitive periodic orbits of $h$ in $J(h)$, and $\lambda(\hat z)$ is the multiplier for $\hat z\in\mathcal{P}$. For $n\ge 1$, we denote $\mathrm{Fix}(h^n)$ the set of all points $z\in J(h)$ with $h^n(z)=z$, and denote $\mathrm{Fix}^\ast(h^n)$ the subset of $\mathrm{Fix}(h^n)$ consisting of primitive periodic points with period $n$. 

We show in the next lemma that multipliers and periods are comparable.

\begin{lemma}\label{lem:period}
Let $h\in\mathbb{C}(z)$ be a hyperbolic rational map of degree at least $2$. Then there exist $0<C_1<C_2$ such that for any $n\ge 1$ and any $z\in\mathrm{Fix}^\ast(h^n)$,  
$$C_1 < \frac{\log|\lambda(\hat z)|}{n} < C_2.$$
\end{lemma}
\begin{proof}
Up to conjugation by M\"obius transformations, we may assume that $\infty\not\in J(h)$. Since $h$ is hyperbolic, there exist $C>1$ and a smooth conformal metric $\rho(z)=\omega(z)dz$ with $\omega(z)>0$ on a neighborhood of $J(h)$ such that for any $z\in J(h)$, we have
$$||h'(z)||_\rho=\frac{\omega(h(z))|h'(z)|}{\omega(z)}>C.$$
For $z\in\mathrm{Fix}^\ast(h^n)$, consider the probability measure 
$$\mu_z=\frac{1}{n}\sum_{j=0}^{n-1}\nu_{h^j(z)},$$
where $\nu_{h^j(z)}$ is the Dirac measure at $h^j(z)$. Then we have
\begin{align*}
\frac{\log|(h^n)'(z)|}{n} &=\int_{J(h)} \log |h'(z)| d\mu_z\\
&=\int_{J(h)} \log||h'(z)||_\rho d\mu_z +\int_{J(h)} \log w(z) - \log w(h(z)) d\mu_z.
\end{align*}
The second term $\int_{J(h)}\log w(z) - \log w(h(z)) d\mu_z = 0$ since $\mu_z$ is invariant under $h_\ast$. It follows that  
$$\frac{\log|(h^n)'(z)|}{n}>\log C.$$

Moreover, since $J(h)$ is compact, there exists $C_0>0$ such that $|h'(z)|<C_0$ for any $z \in J(h)$. It follows that $\log |(h^n)'(z)|<n\log C_0$.

The conclusion follows immediately by setting $C_1=\log C$ and $C_2 = \log C_0$.
\end{proof}

\subsection{Dynamical  L-functions}
Let $h\in\mathbb{C}(z)$ be a hyperbolic rational map of degree at least $2$ and let $\mathcal{P}$ be the set of primitive periodic orbits of $h$ in $J(h)$. Recall that the \emph{Ruelle dynamical zeta function} is given by
\begin{align}\label{equ:zeta1}
\zeta(s):=\prod_{\hat z \in\mathcal{P}}(1-|\lambda(\hat z)|^{-s})^{-1}.
\end{align}
Let $\chi : \mathbb{S}^1 \to \mathbb{S}^1$ be an unitary character. The group of unitary characters of $\mathbb{S}^1$ can be identified with $\mathbb Z$ via the map $\chi_k(x) = x^k$. Then the \emph{dynamical (Hecke) L-function} is obtained by twisting the zeta function by an unitary character:
$$\zeta(s, k):=\prod_{\hat z \in\mathcal{P}}(1-\chi_k(\lambda_\theta(\hat z))|\lambda(\hat z)|^{-s})^{-1}.$$
The (dynamical) $L$-function is holomorphic and non-vanishing on $\mathrm{Re}(s) > \delta$. Moreover, the $L$-function $\zeta(s,k)$ can also be written as
\begin{equation} \label{equ:zeta}
\zeta(s,k)=\exp\left(\sum_{n=1}^\infty\frac{1}{n}\sum_{z \in {\rm Fix}(h^n)}\chi_k\left(\frac{(h^n)'(z)}{|(h^n)'(z)|}\right)|(h^n)'(z)|^{-s}\right).
\end{equation}
We refer the reader to \cite{Oh17} and the references therein for more details on $L$-functions.

%\hmn{It follows that the logarithmic derivative (with respect to $s$) of $\zeta(s,k)$ is
%	\begin{equation}\label{equ:zeta-derivative}
%	-\frac{\zeta'(s,k)}{\zeta(s,k)}=\sum_{n\ge 1}\frac{1}{n}\sum_{z\in\mathrm{Fix}(h^n)}\chi_k\left(\frac{(h^n)'(z)}{|(h^n)'(z)|}\right)\frac{\log |(h^n)'(z)|}{|(h^n)'(z)|^s}.
%	\end{equation}
%	}
The $L$-function $\zeta(s,k)$ satisfies the following analytic property. 

\begin{theorem}[\cite{Oh17}, Theorem 1.2]\label{thm:analytic-L}
Let $h\in\mathbb{C}(z)$ be a hyperbolic rational map of degree at least $2$.
\begin{enumerate}
\item If $h$ is not conjugate to a monomial, then there exists $\epsilon_0>0$ such that $\zeta(s,0)$ is analytic and non-vanishing on $\mathrm{Re}(s) > \delta- \epsilon_0$ except for the simple pole at $s = \delta$.
\item If $J(h)$ is not contained in a circle in $\widehat{\mathbb{C}}$, then for any $k \neq 0$, there exists $\epsilon_0>0$ such that $\zeta(s,k)$ is analytic and non-vanishing on $\mathrm{Re}(s) > \delta - \epsilon_0$.
\end{enumerate}
\end{theorem}

Moreover,  %combining \cite[Proposition 6.3(3) and Theorem 6.4]{Oh17} and \cite[Theorem 4.2]{Ellison85} (also see \cite[Lemma 3]{Pollicott98}), 
we have the following bound for the logarithmic derivative of $\zeta(s,k)$.

\begin{lemma} [\cite{Oh17} Section 6] \label{lem:epsilon1}
Let $h\in\mathbb{C}(z)$ be a hyperbolic rational map of degree at least $2$, and suppose  $J(h)$ is not contained in a circle in $\widehat{\mathbb{C}}$. Then for any $\epsilon>0$, there exist $C_{\epsilon}>1$, $0<\epsilon_1<\epsilon_0$ and $0<\beta<1$ such that for all ${\rm Re}(s)\ge\delta - \epsilon_1$ and $|{\rm Im}(s)| \ge 1$, we have 
\begin{equation} 
\left|\frac{\zeta'(s,k)}{\zeta(s,k)} \right| \le C_{\epsilon}(|k|+1)^{2+\epsilon}|{\rm Im}(s)|^\beta.
\end{equation}
\end{lemma}

Following arguments of \cite{Pollicott98}, we can deduce equation (\ref{equ:number}) from the above lemma. In particular, we have
\begin{equation} \label{equ:eta}
0 < \eta\le \epsilon_1/2.
\end{equation}

\subsection{The Mellin transform and the Fourier transform}\label{sec:mellin}
For $g\in C_c^\infty(0, \infty)$, the \emph{Mellin transform} of $g$ is 
$$\mathcal{M}(g)(s) :=\int_{0}^{\infty}g(x)x^{s-1}dx.$$
Since $g\in C_c^\infty(0, \infty)$, $\mathcal{M}(g)(s)$ is analytic in the $s$-plane.  The Mellin inversion formula allows us to write $g(x)$ in terms of $\mathcal{M}(g)(s)$. More precisely, we have
$$g(x)=\frac{1}{2\pi i}\int_{\mathrm{Re}(s)=c}\mathcal{M}(g)(s)x^{-s}ds$$
for any $c \in \mathbb R$. We refer the reader to \cite{Davies85} and \cite{Oberhettinger74} for more details on the Mellin transform and its inversion.

Recall that the Fourier transform of $g$ is defined as 
$$\mathcal{F}(g)(y) = \hat{g}(y) := \int_{-\infty}^{\infty} g(x)e^{-2\pi i yx}dx.$$
Then the Mellin transform is related to the Fourier transform as follows
$$\mathcal{M}(g)(-2\pi i y) = \mathcal{F}(g(e^{x}))(y).$$

If $g$ has compact support in $(0,1]$, then $\mathcal{M}(g)(s)$ is analytic in the $s$-plane. Moreover, $g\circ \exp$ also has compact support. Therefore $g\circ \exp$ belongs to the Schwartz class (i.e. smooth functions whose derivatives of all orders decay faster than any polynomial). Since the Fourier transform is an isometry between the Schwartz class, $\mathcal{F}(g(e^{x}))(y)$ is in the Schwartz class and thus decays faster than any polynomial as $y \to \pm \infty$. Therefore $\mathcal{M}(g)(-2\pi i y)$ decays faster than any polynomial along the imaginary axis as $y$ goes to infinity.

Now choose $g \in C_c^{\infty}(0,\infty)$ such that $[1/3, 2/3] \subset \mathrm{supp}(g)\subset (0,1]$ and $g(x)>0$ on the interior of its support. Let $\gamma>0$ and consider 
\begin{equation} \label{equ:deftestphi}
\phi(x) := x^{-\gamma}g(x).
\end{equation}
Then $\phi \in C_c^\infty(0,\infty)$ is real-valued and non-negative. Moreover, it satisfies the following conditions:
\begin{enumerate}
	\item[(a.1)] $||\phi||_\infty<\infty$,
	\item[(a.2)] $\mathrm{supp}(\phi)\subset (0,1]$,
	\item[(a.3)] $L_\phi:=\min\limits_{1/3\le x\le 2/3}\phi(x)>0$, 
	\item[(a.4)] the Mellin transform $$\mathcal{M}(\phi)(s) = \mathcal{M}(g)(s-\gamma)$$
	is analytic in the $s$-plane, and
	\item[(a.5)] for $s\in\mathbb{C}$ with $\mathrm{Re}(s) = \gamma$, there exist $A_\phi:=A_\phi(\mathrm{Re}(s))>0$ and $Y>1$ such that for $|\mathrm{Im(s)}|>Y$, we have
	\begin{equation}\label{equ:M}
	|\mathcal{M}(\phi)(s)|\le\frac{A_\phi}{(1+|\mathrm{Im(s)}|)^2}.
	\end{equation}
\end{enumerate}

%The following is an example of such a $\phi$.
%\begin{example}
%	Consider 
%	$$\phi(x)=\begin{cases}
%	-\frac{x}{2}\log\frac{x}{2}, 0 < x \le 2\\
%	0, x >2
%	\end{cases}.$$
%	Then $||\phi||_\infty=e$, $B_\phi=2$, and $L_\phi=\log\sqrt{2}$. Moreover, we have the Mellin transform
%	$$\mathcal{M}(\phi)(s)=\frac{2^s}{(1+s)^2}$$ which is analytic on $Re(s) > -1$.
%	It follows that for $\mathrm{Re}(s)>0$,
%	$$|\mathcal{M}(\phi)(s)|=\frac{2^{\mathrm{Re}(s)}}{|1+s|^2}\le \frac{2^{\mathrm{Re}(s)}}{(1+|\mathrm{Im(s)}|)^2}$$
%	and we can take $A_\phi=2^{\mathrm{Re}(s)}$.
%\end{example}

\section{Definitions of auxiliary functions} \label{sec:defs}
Let $h\in\mathbb{C}(z)$ be a hyperbolic rational map of degree at least $2$. Our main goal of this section is to define two smooth functions $\Phi_{K,t}$ and $\Phi^\ast_{K,t}$ as introduced in Section \ref{sec:intro}. We will use the variance of $\Phi^\ast_{K,t}$ in Section 5 to prove Theorem \ref{thm:main}. These definitions were inspired by the so-called smooth count of Gaussian primes introduced in \cite{Rudnick19}.

\subsection{Window functions for angles}\label{sec:window}
Let $f \in C_c^{\infty}(\mathbb R)$ be even and real-valued with support $[-1/2,1/2]$ such that $f >0$ on $(-1/2,1/2)$. Fix $K \gg 1$ and define
$$F_K(\theta) := \sum_{j \in \mathbb Z} f\left(\frac{K}{2\pi}(\theta-2\pi j)\right).$$
Then $F_K$ is $2\pi$-periodic, even and nonnegative real-valued. Moreover, $F_K(\theta)>0$ if and only if $\theta\in(2\pi j-\pi/K, 2\pi j+\pi/K)$ for some $j\in\mathbb{Z}$.

%\begin{remark}\label{rmk:j}
%For $z \in {\rm Fix}(h^n)$, recall that ${\rm Arg}((h^n)'(z))\in(-\pi,\pi]$ the principle argument of $(h^n)'(z)$. It follows immediately that $F_K(\theta-\mathrm{Arg}((h^n)'(z))) > 0$ if and only if $\theta-{\rm Arg}((h^n)'(z))\in(2\pi j-\pi/K, 2\pi j+\pi/K)$ for some $j\in\mathbb{Z}$.
%\end{remark}

The Fourier expansion of $F_K(\theta)$ is given by
\begin{equation}\label{equ:F}
F_K(\theta)=\sum_{k\in\mathbb{Z}}\widehat F_K(k)e^{ik\theta}=\sum_{k\in\mathbb{Z}}\frac{1}{K}\hat f\left(\frac{k}{K}\right) e^{ik\theta}
\end{equation} 
where $\hat f$ is the Fourier transform of $f$. Note that $\hat{f}(y)$ is also even and real-valued.
For $z\in\mathrm{Fix}(h^n)$ and $k\ge 0$, we denote
$$\Theta_k(z,n):=\chi_k \left(\frac{(h^n)'(z)}{|(h^n)'(z)|}\right)=\left(\frac{(h^n)'(z)}{|(h^n)'(z)|}\right)^k.$$

\begin{lemma}\label{lem:FourierF}
For $z\in\mathrm{Fix}(h^n)$, 
\begin{equation}\label{equ:F-theta}
F_K(\theta-\mathrm{Arg}((h^n)'(z)))=F_K(\mathrm{Arg}((h^n)'(z))-\theta)=\sum_{k\in\mathbb{Z}}\frac{1}{K}\hat f\left(\frac{k}{K}\right) e^{-ik\theta}\Theta_k(z,n)
\end{equation}
and therefore
	$$\int_{-\infty}^{+\infty}F_K(\theta-\mathrm{Arg}((h^n)'(z)))\frac{d\theta}{2\pi}=\frac{\hat f(0)}{K}.$$ 
\end{lemma}
\begin{proof}
Equation (\ref{equ:F-theta}) follows from the Fourier expansion (\ref{equ:F}) of $F_K$. 
For $z \in {\rm Fix}(h^n)$, ${\rm Arg}((h^n)'(z))\in(-\pi,\pi]$. Then $F_K(\theta-\mathrm{Arg}((h^n)'(z))) > 0$ if and only if $\theta-{\rm Arg}((h^n)'(z))\in(2\pi j-\pi/K, 2\pi j+\pi/K)$ for some unique $j\in\mathbb{Z}$. Therefore for such $j$, we have
$$\int_{-\infty}^{+\infty}F_K(\theta-\mathrm{Arg}((h^n)'(z)))\frac{d\theta}{2\pi}=\int_{2j\pi-\pi}^{2j\pi+\pi}F_K(u)\frac{du}{2\pi}.$$
It follows that 
$$\int_{-\infty}^{+\infty}F_K(\theta-\mathrm{Arg}((h^n)'(z)))\frac{d\theta}{2\pi}=\int_{2j\pi-\pi}^{2j\pi+\pi}\sum_{k\in\mathbb{Z}}\frac{1}{K}\hat f\left(\frac{k}{K}\right) e^{-iku}\frac{du}{2\pi}=\frac{\hat f(0)}{K}.$$
\end{proof}

\subsection{Smooth counting functions}\label{sec:function}
In this subsection, we define smooth functions $\Phi_{K,t}(\theta)$ and $\Phi^*_{K,t}(\theta)$ as stated in Section \ref{sec:intro}. We will estimate their variance in the next section.

Fix $\gamma>0$ and pick $\phi=\phi_\gamma \in C_c^{\infty}(0,\infty)$ which satisfies $(a.1)-(a.5)$ in Section \ref{sec:mellin}. Let $f$ and $F_K$ be as in Section \ref{sec:window}.
%Let $\phi \in C^{\infty}(0,\infty)$ be a real-valued and non-negative function.
%For $K \gg 1$ and $t \gg 1$, we define
%$$\bar\Phi_{K,t}(\theta) :=  \sum_{n \ge1} \frac{1}{n} \sum_{z \in\mathrm{Fix}(h^n)}\phi \left(\frac{|(h^n)'(z)|}{t}\right)F_K\left(  {\rm Arg}(h^n)'(z)-\theta\right)$$
%and
%$$\bar\Phi^*_{K,t}(\theta) :=\sum_{n \ge1} \frac{1}{n} \sum_{z \in\mathrm{Fix}^*(h^n)}\phi \left(\frac{|(h^n)'(z)|}{t}\right) F_K\left(  {\rm Arg}(h^n)'(z)-\theta\right).$$
%If $\phi$ is sufficiently close to the characteristic function on $(0,1)$ and $f$ is sufficiently close to the characteristic function on $(-1/2,1/2)$, then the function $\bar\Phi^*_{K,t}(\theta)$ is a smooth count for the number of primitive periodic orbits with multipliers whose norm less than $t$ and angles lying in a window of scale $2\pi/K$ around $\theta$. 
For $K\gg 1$ and $t>1$, we define
\begin{equation} \label{equ:defPhi}
\Phi_{K,t}(\theta) :=  \sum_{n \ge1} \frac{1}{n} \sum_{z \in\mathrm{Fix}(h^n)}\phi \left(\frac{|(h^n)'(z)|}{t}\right)F_K\left(  {\rm Arg}(h^n)'(z)-\theta\right)\log|(h^n)'(z)|,
\end{equation}
and
\begin{equation} \label{equ:defPhistar}
\Phi^*_{K,t}(\theta) :=\sum_{n \ge1} \frac{1}{n} \sum_{z \in\mathrm{Fix}^*(h^n)}\phi \left(\frac{|(h^n)'(z)|}{t}\right) F_K\left(  {\rm Arg}(h^n)'(z)-\theta\right)\log|(h^n)'(z)|.
\end{equation}
By substituting the Fourier expansion (\ref{equ:F}) of $F_K$, we can rewrite $\Phi_{K,t}(\theta)$ and $\Phi^\ast_{K,t}(\theta)$ as
\begin{equation} \label{equ:Phi}
\Phi_{K,t}(\theta) = \sum_{k \in \mathbb Z} e^{-i\theta k}\frac{1}{K} \hat{f}\left(\frac{k}{K}\right)\sum_{n \ge 1} \frac{1}{n} \sum_{z \in\mathrm{Fix}(h^n)} \phi \left(\frac{|(h^n)'(z)|}{t}\right)\Theta_k(z,n)\log|(h^n)'(z)|,
\end{equation}
and 
\begin{equation} \label{equ:Phistar}
\Phi_{K,t}^\ast(\theta) = \sum_{k \in \mathbb Z} e^{-i\theta k}\frac{1}{K} \hat{f}\left(\frac{k}{K}\right)\sum_{n \ge 1} \frac{1}{n} \sum_{z \in\mathrm{Fix}^\ast(h^n)} \phi \left(\frac{|(h^n)'(z)|}{t}\right)\Theta_k(z,n)\log|(h^n)'(z)|.
\end{equation}

To ease notations, for $k\in\mathbb{Z}$ we set
\begin{equation}\label{def:Lambda}
\Lambda_{t,n,k}(\theta):=\frac{1}{n}\sum_{z\in\mathrm{Fix}(h^n)}\phi\left(\frac{|(h^n)'(z)|}{t}\right)\Theta_k(z,n)\log|(h^n)'(z)|,
\end{equation}
and 
\begin{equation}\label{def:Lambdastar}
\Lambda^{\ast}_{t,n,k}(\theta):=\frac{1}{n}\sum_{z\in\mathrm{Fix}^\ast(h^n)}\phi\left(\frac{|(h^n)'(z)|}{t}\right)\Theta_k(z,n)\log|(h^n)'(z)|.
\end{equation}
Then we have 
\begin{equation}\label{def:Phi}
\Phi_{K,t}(\theta)=\sum_{k\in\mathbb{Z}}\frac{1}{K}\hat f\left(\frac{k}{K}\right)e^{-ik\theta}\sum_{n\ge 1}\Lambda_{t,n,k}(\theta),
\end{equation}
and
\begin{equation}\label{def:Phistar}
\Phi^\ast_{K,t}(\theta)=\sum_{k\in\mathbb{Z}}\frac{1}{K}\hat f\left(\frac{k}{K}\right)e^{-ik\theta}\sum_{n\ge 1}\Lambda^\ast_{t,n,k}(\theta).
\end{equation}

\section{The variance of $\Phi^*_{K,t}(\theta)$} \label{sec:var}
In this section, we let $h\in\mathbb{C}(z)$ be a hyperbolic rational map of degree at least $2$ with Julia set $J(h)$ not contained in a circle in $\widehat{\mathbb C}$. We continue to use the notation from the previous sections. 
For any $\epsilon>0$, let $\epsilon_1>0$ be as in Lemma \ref{lem:epsilon1} and set $\gamma=\delta-\epsilon_1$. We consider the corresponding functions $\Phi_{K,t}(\theta)$ and $\Phi^*_{K,t}(\theta)$ in Section \ref{sec:function}. Fixing $K \gg 1$, we prove the following theorem regarding estimates on the variance of $\Phi^*_{K,t}(\theta)$ for sufficiently large $t$. 
\begin{theorem}\label{prop:V}
	For any $0<\epsilon<1$, there exists $M_\epsilon>0$ such that
	$$\mathrm{Var}(\Phi^\ast_{K,t})<M_\epsilon\max\left\{\frac{t^{\delta}(\log t)^2}{K}, K^5t^{2\delta-2\epsilon_1}\right\}.$$
\end{theorem}
Before we proceed to the proof of the theorem, we observe that
\begin{align*}
&\mathrm{Var}(\Phi^\ast_{K,t}) :=||\Phi^\ast_{K,t}-\mathbb{E}(\Phi^\ast_{K,t})||_{L^2}\\
&\le||\Phi^\ast_{K,t}-\Phi_{K,t}||_{L^2}+||\Phi_{K,t}-\mathbb{E}(\Phi_{K,t})||_{L^2}+|\mathbb{E}(\Phi^\ast_{K,t})-\mathbb{E}(\Phi_{K,t})|
\end{align*}
where $||\cdot||_{L^2}$ denotes the $L^2$-norm.
Subsections 4.1-4.3 are devoted to estimating the three terms in the last line of the inequality. We prove Theorem \ref{prop:V} in the last subsection.

We begin with the third term involving expected values.
\subsection{Expected values: the third term}
We first estimate the expected value of $\Phi^\ast_{K,t}$.  
We denote $g_1(x)\asymp g_2(x)$ if there exist $0<A'<A''$ and $x_0\in\mathbb{R}$ such that $A'g_2(x)< g_1(x)<A''g_2(x)$ for all $x>x_0$. 
\begin{proposition}\label{prop:E}
	The expected value of $\Phi^\ast_{K,t}$ is
	$$\mathbb{E}(\Phi^\ast_{K,t})\asymp\frac{t^\delta\log t}{K}.$$
	%and 
	%$$\frac{t^\delta}{K(\log t)^2}=O(\mathbb{E}(\Phi^\ast_{K,t,\epsilon})).$$
\end{proposition}
\begin{proof}
	By definition of expected value, we have
	$$\mathbb{E}(\Phi^\ast_{K,t})=\frac{1}{K}\hat f(0)\sum_{n\ge 1}\Lambda^\ast_{t,n,0}.$$
	Note that for any fixed $n \ge 1$, equation (\ref{equ:number}) implies that
	\begin{equation*}\label{equ:number1}
	\#\{z\in\mathrm{Fix}^\ast(h^n):|(h^n)'(z)|\le t\}\sim n\mathrm{Li}(t^{\delta}).
	\end{equation*} 
	Then we obtain
	%By (\ref{equ:number1}), we have
	$$\Lambda^\ast_{t,n,0}:=\frac{1}{n}\sum_{z\in\mathrm{Fix}^\ast(h^n)}\phi\left(\frac{|(h^n)'(z)|}{t}\right)\log|(h^n)'(z)|<2||\phi||_\infty t^{\delta}.$$
	%Moreover, there exists $A>0$ such that for $t$ sufficiently large,
	%$$\Lambda^\ast_{t,n,0,\epsilon}:=\frac{1}{n}\sum_{z\in\mathrm{Fix}^\ast(h^n)}\phi_{0,\epsilon}\left(\frac{|(h^n)'(z)|}{t}\right)\log^+(|(h^n)'(z)|)\ge\frac{1}{2n}L_\phi\sum_{z\in\mathrm{Fix}^\ast(h^n)}1\ge A\frac{t^\delta}{\log t}$$
%Since log of the norm of multipliers and periods are comparable by Lemma \ref{lem:period}, 
Let $C_1>0$ be as in Lemma \ref{lem:period}. Then 
	$$\mathbb{E}(\Phi^\ast_{K,t})\le|\hat f(0)|\left|\frac{1}{K}\sum_{n=1}^{(\log t)/C_1}\Lambda^\ast_{t,n,0}\right|<\frac{2}{C_1}|\hat f(0)|||\phi||_\infty\frac{t^{\delta}\log t}{K}.$$
	
	Moreover, by equation (\ref{equ:number}), there exists $C'=C'(h)>0$ such that 
	\begin{equation}\label{equ:half}
	\#\left\{\hat z\in\mathcal{P}: \frac{t}{3}\le|\lambda(\hat z)|<\frac{2t}{3}\right\}\ge \mathrm{Li}\left(\frac{(2t)^\delta}{3^\delta}\right)-\left(\frac{3}{2}\right)^\delta\mathrm{Li}\left(\frac{t^\delta}{3^\delta}\right)\ge C'\mathrm{Li}(t^\delta).
	\end{equation}
	It follows that 
	$$\Lambda^\ast_{t,n,0} \ge\frac{1}{n}\sum_{\substack{z\in\mathrm{Fix}^\ast(h^n)\\ \frac{t}{3}\le|\lambda(\hat z)|<\frac{2t}{3}}}\phi\left(\frac{|(h^n)'(z)|}{t}\right)\log|(h^n)'(z)|\ge L_\phi  C' t^\delta.$$
	Let $C_2>0$ be as in Lemma \ref{lem:period}. Inequality (\ref{equ:half}) implies that $n\ge C_2\log(t/3)$. Hence 
$$\mathbb{E}(\Phi^\ast_{K,t})\ge|\hat f(0)|\left|\frac{1}{K}\sum_{n=1}^{C_2\log(t/3)}\Lambda^\ast_{t,n,0}\right|\ge\frac{1}{K}|\hat f(0)|\cdot C_2\log(t/3)\cdot L_\phi C' t^\delta.$$
The conclusion follows.
\end{proof}

The following result bounds the difference of the expected values of $\Phi_{K,t}$ and $\Phi^\ast_{K,t}$.
\begin{lemma}\label{lem:diff-E}
	There exists $B=B(h,\phi,f)>0$ such that % difference of the expected values of $\Phi_{K,t}$ and $\Phi^\ast_{K,t}$ is
	$$|\mathbb{E}(\Phi_{K,t})-\mathbb{E}(\Phi^\ast_{K,t})|< B \frac{t^{\delta/2}\log t}{K}.$$
\end{lemma}
\begin{proof}
	By definition of expected value, we write
	$$\mathbb{E}(\Phi_{K,t})-\mathbb{E}(\Phi^\ast_{K,t})=\frac{1}{K}\hat f(0)\sum_{n\ge 1}(\Lambda_{t,n,0}-\Lambda^\ast_{t,n,0}).$$
	Note that for any fixed $n \ge 1$, equation (\ref{equ:number}) implies that
	\begin{equation*}\label{equ:number2}
	\#\{z\in\mathrm{Fix}(h^n)\setminus\mathrm{Fix}^\ast(h^n):|(h^n)'(z)|\le t\}<n\mathrm{Li}(t^\frac{\delta}{2}).
	\end{equation*} 
	Indeed, for $z\in\mathrm{Fix}(h^n)\setminus\mathrm{Fix}^\ast(h^n)$ with $|(h^n)'(z)|\le t$, we have $\lambda(\hat z)<t^{1/2}$ and $\#\hat z\le n/2$.
	It follows that
	%By (\ref{equ:number2}), we have
	\begin{align*}
	|\Lambda_{t,n,0}-\Lambda^\ast_{t,n,0}| &=\frac{1}{n}\sum_{z\in\mathrm{Fix}(h^n)\setminus\mathrm{Fix}^\ast(h^n)}\phi\left(\frac{|(h^n)'(z)|}{t}\right)\log|(h^n)'(z)|\\
	&< 2||\phi||_\infty t^\frac{\delta}{2}.
	\end{align*}
Let $C_1>0$ be as in Lemma \ref{lem:period}. Then
$$|\mathbb{E}(\Phi_{K,t})-\mathbb{E}(\Phi^\ast_{K,t})|\le \frac{1}{K}|\hat f(0)|\sum_{n=1}^{(\log t)/C_1}(|\Lambda_{t,n,0}-\Lambda^\ast_{t,n,0}|)< \frac{1}{C_1}||\phi||_\infty|\hat f(0)|\frac{t^{\delta/2}\log t}{K}.$$
\end{proof}

As a corollary of the previous two results, we obtain an estimate on $\mathbb{E}(\Phi_{K,t})$.
\begin{corollary}
	The expected value of $\Phi_{K,t}$ satisfies
	$$\mathbb{E}(\Phi_{K,t})\asymp\frac{t^\delta\log t}{K}.$$
\end{corollary}

\subsection{Variance of $\Phi_{K,t}$: the second term}
Before we prove the main result of this subsection, namely Proposition \ref{prop:term2} , we need a lemma which bounds the term $\sum_{n\ge 1}\Lambda_{t,n,k}(\theta)$. Recall that $\epsilon_0>0$ is as in Theorem \ref{thm:analytic-L}.

\begin{lemma}\label{lem:total}
	For any $\epsilon>0$, if $k\not=0$, there exists $C'_\epsilon:=C'_\epsilon(h,\phi)>0$ such that
	$$\left|\sum_{n\ge 1}\Lambda_{t,n,k}(\theta)\right|<C'_\epsilon (1+|k|)^{2+\epsilon}t^{\delta-\epsilon_1}.$$
\end{lemma}
\begin{proof}
	By definition,
	$$\sum_{n \ge 1}\Lambda_{t,n,k}(\theta):= \sum_{n \ge 1}\frac{1}{n}\sum_{z\in\mathrm{Fix}(h^n)} \phi\left(\frac{|(h^n)'(z)|}{t}\right)\Theta_k(z,n)\log|(h^n)'(z)|.$$
	By the assumption on $\phi$, the line $\mathrm{Re}(s) = \delta - \epsilon_1$ is in the domain of holomorphy of the Mellin transform $\mathcal{M}(\phi)(s)$ for $\phi$. Then we apply Mellin inversion and obtain that
	%Applying Mellin inversion formula to $\phi$ and noting that for any $0<\epsilon'<\epsilon_0$, the line $\mathrm{Re}(s) = \delta - \epsilon'$ is in the domain of holomorphy of $\mathcal{M}(\phi)(s)$, we have
	$$\sum_{n\ge 1}\Lambda_{t,n,k}(\theta)=\frac{1}{2\pi i}\int_{\mathrm{Re}(s)=\delta-\epsilon_1}\sum_{n \ge 1}\frac{1}{n}\sum_{z\in\mathrm{Fix}(h^n)}\Theta_k(z,n)\log|(h^n)'(z)|\frac{t^s}{|(h^n)'(z)|^s}\mathcal{M}(\phi)(s)ds.$$
	On the other hand, by formula (\ref{equ:zeta}) we calculate the logarithmic derivative (with respect to $s$) of the $L$-function $\zeta(s,k)$:
	$$-\frac{\zeta'(s,k)}{\zeta(s,k)}=\sum_{n\ge 1}\frac{1}{n}\sum_{z\in\mathrm{Fix}(h^n)}\Theta_k(z,n)\frac{\log |(h^n)'(z)|}{|(h^n)'(z)|^s}.$$
	Since $k\not=0$, by Theorem  \ref{thm:analytic-L}(2), the function $\frac{\zeta'(s,k)}{\zeta(s,k)}$ is analytic on the half plane $\mathrm{Re}(s)\ge\delta-\epsilon_0$. 
	Therefore, it follows that $$\sum_{n\ge 1}\Lambda_{t,n,k}(\theta)=\frac{1}{2\pi i}\int_{\mathrm{Re}(s)=\delta-\epsilon_1}-\frac{\zeta'(s,k)}{\zeta(s,k)}t^s\mathcal{M}(\phi)(s)ds.$$
	Thus, to bound $\left|\sum_{n\ge 1}\Lambda_{t,n,k}(\theta)\right|$, it suffices to bound $\int_{\mathrm{Re}(s)=\delta-\epsilon_1} \left|\frac{\zeta'(s,k)}{\zeta(s,k)}t^s\mathcal{M}(\phi)(s) \right||ds|$.

	 Write $s=\delta-\epsilon_1+yi$.  By Lemma \ref{lem:epsilon1}, there exists $C_\epsilon>0$ such that for $|\mathrm{Im}(s)|\ge 1$, 
		$$\left|-\frac{\zeta'(s,k)}{\zeta(s,k)}\right|\le C_\epsilon(1+|k|)^{2+\epsilon}|y|^\beta.$$
Moreover, by \eqref{equ:M}, there exist $A_\phi>0$ and $Y>1$ such that for $|y|>Y$,
$$\left|\mathcal{M}(\phi)(s)\right| \le \frac{A_\phi}{(1+|y|)^2}.$$
Therefore, we obtain
\begin{align*}
\left|\sum_{n\ge 1}\Lambda_{t,n,k}(\theta)\right|&\le \frac{1}{2\pi} \int_{\mathrm{Re}(s)=\delta-\epsilon_1} \left|\frac{\zeta'(s,k)}{\zeta(s,k)}t^s\mathcal{M}(\phi)(s) \right||ds|\\
& =\frac{1}{2\pi} t^{\delta-\epsilon_1} \int_{-\infty}^\infty \left|\frac{\zeta'(\delta-\epsilon_1+iy,k)}{\zeta(\delta-\epsilon_1+iy,k)}\mathcal{M}(\phi)(\delta-\epsilon_1+iy) \right| dy \\
& \le \frac{t^{\delta-\epsilon_1}}{2\pi} \int_{|y|\le Y} \left|\frac{\zeta'(\delta-\epsilon_1+iy,k)}{\zeta(\delta-\epsilon_1+iy,k)}\mathcal{M}(\phi)(\delta-\epsilon_1+iy) \right| dy\\
&\ \ \ \ + \frac{t^{\delta-\epsilon_1}}{2\pi} C_{\epsilon}(1+|k|)^{2+\epsilon}\int_{|y|> Y} \frac{A_\phi |y|^\beta}{(1+|y|)^2} dy.
\end{align*}
Note that $\left|\frac{\zeta'(\delta-\epsilon_1+iy,k)}{\zeta(\delta-\epsilon_1+iy,k)}\mathcal{M}(\phi)(\delta-\epsilon_1+iy) \right|$ is analytic in $y$. It follows that 
$$\int_{|y|\le Y} \left|\frac{\zeta'(\delta-\epsilon_1+iy,k)}{\zeta(\delta-\epsilon_1+iy,k)}\mathcal{M}(\phi)(\delta-\epsilon_1+iy) \right| dy<\infty.$$
%$\le C_{\epsilon}' t^{\delta-\epsilon'}(1+|k|)^{2+\epsilon}$
%for some $C'_\epsilon:=C'_\epsilon(h,\phi)>0$. 
Then the conclusion follows.
%	By assumptions on $\phi$ and definition of $\phi_{k,\epsilon}$, there exists $C'_\epsilon>0$ such that
%	$$\left|\sum_{n\ge 1}\Lambda_{t,n,k,\epsilon}(\theta)\right|\le \frac{C'_\epsilon t^{\delta-\epsilon'}}{B_\phi^{\delta-\epsilon'}}\int_{-\infty}^\infty\frac{1}{(1+|\mathrm{Im}(s)|)^{2-\beta}}d\mathrm{Im}(s)=2(1-\beta)\frac{C'_\epsilon}{B_\phi^{\delta-\epsilon'}} t^{\delta-\epsilon'}.$$
\end{proof}

Now we give a bound on the variance of $\Phi_{K,t}$.
\begin{proposition}\label{prop:term2}
	For any $0<\epsilon<1$, there exists $B_{\epsilon}:=B_\epsilon(h,\phi,f)>0$ such that
	$$\mathrm{Var}(\Phi_{K,t})<B_\epsilon K^5 t^{2\delta-2\epsilon_1}.$$
\end{proposition}
\begin{proof}
	Note that
	\begin{align*}
	\Phi_{K,t}(\theta)-\mathbb{E}(\Phi_{K,t})&=\sum_{k\not=0}\frac{1}{K}\hat f\left(\frac{k}{K}\right)e^{-ik\theta}\sum_{n\ge 1}\Lambda_{t,n,k}(\theta)\\
	&=\sum_{k\not=0}i\frac{K^2}{k^3}\widehat{f^{(3)}}\left(\frac{k}{K}\right)e^{-ik\theta}\sum_{n\ge 1}\Lambda_{t,n,k}(\theta),
	\end{align*}
	where $\widehat{f^{(3)}}$ is the Fourier transform of the third derivative $f^{(3)}$ of $f$.
	It follows that 
	\begin{align*}
	\mathrm{Var}(\Phi_{K,t,}) &:=||\Phi_{K,t}(\theta)-\mathbb{E}(\Phi_{K,t})||_{L^2}  \\ & \le\sum_{k\not=0}\frac{K^4}{k^6}\left|\widehat{f^{(3)}}\left(\frac{k}{K}\right)\right|^2\left|\sum_{n\ge 1}\Lambda_{t,n,k}(\theta)\right|^2.
	\end{align*}

We claim that the series $\sum_{k \in \mathbb Z}\frac{1}{K}\left|\widehat{f^{(3)}}\left(\frac{k}{K}\right)\right|^2$ is finite. Indeed, we observe that 
$$\sum_{k \in \mathbb Z}\frac{1}{L}\left|\widehat{f^{(3)}}\left(\frac{k}{L}\right)\right|^2\to\int_{-\infty}^{+\infty}|\widehat{f^{(3)}}(y)|^2dy$$ 
as $L \to \infty$. Since $f^{(3)} \in C_c^{\infty}(\mathbb R)$, by Parseval's identity, 
$$\int_{-\infty}^{+\infty}|\widehat{f^{(3)}}(y)|^2dy = \int_{-\infty}^{+\infty}|f^{(3)}(x)|^2dx<\infty.$$%The integral on the right hand side is convergent 

Applying Lemma \ref{lem:total} and noting that $\sum_{k\not=0}(1+|k|^{2+\epsilon})^2/k^6<\infty$, we obtain the desired bound.
\end{proof}

\subsection{Primitive versus non-primitive: the first term}
The main result of this subsection is Lemma \ref{lem:term1} where we obtain a bound on the first term $\mathbb{E}(|\Phi^\ast_{K,t}-\Phi_{K,t}|^2)^{1/2}$. Similar to the previous subsection, we first prove a lemma bounding $\sum_{n\ge 1}\Lambda_{t,n,k}(\theta)-\sum_{n\ge 1}\Lambda^\ast_{t,n,k}(\theta)$.
\begin{lemma}\label{lem:diff}
	There exists $B':=B'(h,\phi,f)>0$ such that for all $k\in\mathbb{Z}$
	$$\left|\sum_{n\ge 1}\Lambda_{t,n,k}(\theta)-\sum_{n\ge 1}\Lambda^\ast_{t,n,k}(\theta)\right|<B't^{\delta/2}\log t.$$
\end{lemma}
\begin{proof}
	Straightforward computations give an upper bound of the left hand side as follows.
	\begin{align*}
	\left|\sum_{n\ge 1}\Lambda_{t,n,k}(\theta)-\sum_{n\ge 1}\Lambda^\ast_{t,n,k}(\theta)\right| &\le \sum_{n\ge 1}\frac{1}{n}\log t\sum_{z\in\mathrm{Fix}(h^n)\setminus\mathrm{Fix}^\ast(h^n)}\phi\left(\frac{|(h^n)'(z)|}{t}\right)\\
	&\le\log t\sum_{n\ge 1}\sum_{z\in\mathrm{Fix}(h^n)\setminus\mathrm{Fix}^\ast(h^n)}\phi\left(\frac{|(h^n)'(z)|}{t}\right)\\
	&\le ||\phi||_\infty(\log t)\cdot\frac{n}{2}\cdot\#\{\hat x\in\mathcal{P}: |\lambda(\hat x)|\le t^{1/2}\}.
	\end{align*}
	%The last inequality follows from {\color{red} details}...
Let $C_1>0$ be as in Lemma \ref{lem:period}. Then $n\le(\log t)/C_1$. Then applying equation (\ref{equ:number}), we obtain the conclusion.
\end{proof}

\begin{lemma}\label{lem:term1}
	There exists $B'':=B''(h,\phi,f)>0$ such that 
	$$\mathbb{E}(|\Phi_{K,t}-\Phi^\ast_{K,t}|^2)<B''\frac{t^\delta(\log t)^2}{K}.$$
\end{lemma}
\begin{proof}
Note that	
$$\Phi^\ast_{K,t}-\Phi_{K,t}=\sum_{k\in\mathbb{Z}}\frac{1}{K}\hat f\left(\frac{k}{K}\right)e^{-ik\theta}\left(\sum_{n\ge 1}\Lambda_{t,n,k}(\theta)-\sum_{n\ge 1}\Lambda^\ast_{t,n,k}(\theta)\right).$$
	Then by Lemma \ref{lem:diff}, there exists $B'>0$ such that
	\begin{align*}
	||\Phi^\ast_{K,t}-\Phi_{K,t}||_{L^2}&=\left|\left|\sum_{k\in\mathbb{Z}}\frac{1}{K}\hat f\left(\frac{k}{K}\right)e^{-ik\theta}\left(\sum_{n\ge 1}\Lambda_{t,n,k}(\theta)-\sum_{n\ge 1}\Lambda^\ast_{t,n,k}(\theta)\right)\right|\right|_{L^2}\\
	&=\left(\sum_{k\in\mathbb{Z}}\frac{1}{K^2}\left|\hat f\left(\frac{k}{K}\right)\right|^2\left|\sum_{n\ge 1}\Lambda_{t,n,k}(\theta)-\sum_{n\ge 1}\Lambda^\ast_{t,n,k}(\theta)\right|^2\right)^{1/2}\\
	&\le\left(\sum_{k\in\mathbb{Z}}\frac{1}{K^2}\left|\hat f\left(\frac{k}{K}\right)\right|^2(B't^{\delta/2}\log t)^2\right)^{1/2}.
	\end{align*}
	It follows that 
	$$||\Phi^\ast_{K,t}-\Phi_{K,t}||_{L^2}<2||\hat f||_{L^2}B' \frac{t^{\delta/2}\log t}{K^{1/2}}.$$
	Thus the conclusion follows.
\end{proof}

\subsection{Proof of Theorem \ref{prop:V}}
We combine the results from the previous three subsections to give a proof of Theorem \ref{prop:V}.
\begin{proof}[Proof of Theorem \ref{prop:V}]
	By triangle inequality, we have
	$$||\Phi^\ast_{K,t}-\mathbb{E}(\Phi^\ast_{K,t})||_{L^2}\le ||\Phi^\ast_{K,t}-\Phi_{K,t}||_{L^2}+||\Phi_{K,t,\epsilon}-\mathbb{E}(\Phi_{K,t})||_{L^2}+|\mathbb{E}(\Phi^\ast_{K,t})-\mathbb{E}(\Phi_{K,t})|.$$
	By Lemma \ref{lem:term1}, we have
	$$||\Phi^\ast_{K,t}-\Phi_{K,t}||_{L^2}=\mathbb{E}(|\Phi^\ast_{K,t}-\Phi_{K,t}|^2)^{1/2}<B''^{1/2}\frac{t^{\delta/2}\log t}{K^{1/2}}.$$
	By Proposition \ref{prop:term2}, we have 
	$$ ||\Phi^\ast_{K,t}-\Phi_{K,t}||_{L^2}=(\mathrm{Var}(\Phi_{K,t}))^{1/2}<B_\epsilon^{1/2} K^{5/2} t^{\delta-\epsilon'}.$$
	By Lemma \ref{lem:diff-E}, we have 
	$$|\mathbb{E}(\Phi_{K,t})-\mathbb{E}(\Phi^\ast_{K,t})|<B\frac{t^{\delta/2}\log t}{K}.$$
	Since we fix $K\gg 1$, we have that 
	$$B\frac{t^{\delta/2}\log t}{K}<B''^{1/2}\frac{t^{\delta/2}\log t}{K^{1/2}}.$$
	Now set $M_\epsilon^{1/2}=B''^{1/2}+B_\epsilon^{1/2}+B$. Then we have
	$$||\Phi^\ast_{K,t}-\mathbb{E}(\Phi^\ast_{K,t})||_{L^2}<M_\epsilon^{1/2}\max\left\{\frac{t^{\delta/2}\log t}{K^{1/2}}, K^{5/2} t^{\delta-\epsilon'}\right\}.$$
	Hence the conclusion follows.
\end{proof}

\section{Proof of Theorem \ref{thm:main}}  \label{sec:proofMainThm}
In this section, we prove the following theorem from which Theorem \ref{thm:main} follows. 

\begin{theorem}\label{thm:51}
	Let $h\in\mathbb{C}(z)$ be a hyperbolic rational map of degree at least $2$. Suppose that $J(h)$ is not contained in a circle in $\widehat{\mathbb{C}}$. For any $0<\epsilon<1$, let $\epsilon_1>0$ be as in Lemma \ref{lem:epsilon1}.
Then for almost all $\theta \in (-\pi, \pi]$, the interval $\left[\theta -\frac{\pi}{K},\theta + \frac{\pi}{K} \right]$ contains at least one multiplier angle ${\rm Arg}(\lambda(\hat{z}))$ with 
$$|\lambda(\hat{z})| \le K^{\frac{7}{2\alpha'}}$$
where $\alpha'= \min \left\{\frac{\delta}{2}, \epsilon_1 \right\}$.
\end{theorem}
To prove the theorem, we apply the estimates in Section \ref{sec:var}, together with the Chebyshev inequality to first show that with an appropriately chosen $t$ in terms of $K$, $\Phi^\ast_{K,t}(\theta)$ is non-zero. We then deduce that for almost all $\theta \in (-\pi,\pi]$, interval $[\theta-\pi/K,\theta+\pi/K]$ contains at least one angle by looking at the window function $F_K(\theta)$.

\begin{proof}
%Pick an $\epsilon_0=\epsilon_0(h)>0$ satisfying (\ref{equ:epsilon}). 
%For any $\epsilon>0$, let $\epsilon'>0$ be as in Lemma \ref{lem:epsilon1}. 
For any $0<\epsilon<1$, by the Chebyshev inequality, Theorem \ref{prop:V} and Proposition \ref{prop:E}, there exists $\widetilde C_\epsilon>0$ such that
\begin{align*}
\mathrm{Prob}\{\theta: |\Phi^\ast_{K,t}(\theta)-\mathbb{E}(\Phi^\ast_{K,t})|>\frac{1}{2}\mathbb{E}(\Phi^\ast_{K,t})\}&\le\frac{4\mathrm{Var}(\Phi^\ast_{K,t})}{\mathbb{E}(\Phi^\ast_{K,t})^2}\\
&<\widetilde{C}_\epsilon\max\left\{\frac{K}{t^{\delta}},\frac{K^7}{t^{2\epsilon_1}(\log t)^2}\right\}.
\end{align*}
%We set
%$$t=\max\left\{K^{1/\delta}(\log K)^{1/\delta}, K^{7/2\epsilon'}\right\}.$$
%It follows that 

For $t=K^{7/2\alpha'},$ we have
$$\mathrm{Prob}\{\theta: |\Phi^\ast_{K,t}(\theta) - \mathbb{E}(\Phi^\ast_{K,t})|>\frac{1}{2}\mathbb{E}(\Phi^\ast_{K,t})\}<\frac{1}{\log K}.$$
Hence for almost all $\theta \in (-\pi,\pi]$, we have
$$\Phi^\ast_{K,t}(\theta)  \ge\frac{1}{2}\mathbb{E}(\Phi^\ast_{K,t}),$$
which combined with Proposition \ref{prop:E}, implies that $\Phi^\ast_{K,t}(\theta)$ is nonzero.
It follows that there exist $n\ge 1$ and $z\in\mathrm{Fix}^\ast(h^n)$ with $|\lambda(\hat z)|\le t$ such that 
$$F_K(\theta-\mathrm{Arg}((h^n)'(z))\not=0.$$
Thus the angle of $\lambda(\hat{z})$ lies in the interval $[\theta-\pi/K,\theta+\pi/K]$.
%Now we consider $\epsilon_0=\alpha(h)$. Then 
%$$K^{1/\delta}(\log K)^{6/\delta+o(1)}=o(K^{1/2\epsilon'}(\log K)^{2/\epsilon'+o(1)}).$$
%By Remark \ref{rmk:epsilon'}, as $\epsilon\to 0$, we have that $\epsilon'\to\epsilon_0$. Thus for any $0<\eta<\epsilon_0$, there exists $\epsilon>0$ such that $\epsilon'<\eta$. Now we fix such $\epsilon$ and $\epsilon'$. Then we have $t<K^{1/2\eta}$. This completes the proof.
\end{proof}

\begin{proof}[Proof of Theorem \ref{thm:main}]
	Theorem \ref{thm:main} follows from Theorem \ref{thm:51} since $\eta < \epsilon_1/2$ (see equation (\ref{equ:eta})).
\end{proof}

\section{Variance of $\mathcal{N}_{K,t}$ when $\mathcal{N}_t \ll K$.}\label{sec:trivial}
In this section, we supplement our result by considering the variance of $\mathcal{N}_{K,t}$ in the case when $\mathcal{N}_t \ll K$. Recall that $\mathcal{N}_t:=\#\{\hat z\in\mathcal{P}:|\lambda(\hat z)|<t\}$.

\begin{proposition}
Let $h\in\mathbb{C}(z)$ be a hyperbolic rational map of degree at least $2$, and suppose $J(h)$ is not contained in a circle in $\widehat{\mathbb{C}}$. If $\mathcal{N}_t \ll K$, then 
$$\mathrm{Var}(\mathcal{N}_{K,t})\sim\frac{\mathcal{N}_t}{K}.$$
\end{proposition}
\begin{proof}
By definition, $$\mathrm{Var}(\mathcal{N}_{K,t})=\mathbb{E}(\mathcal{N}_{K,t}^2)-\mathbb{E}(\mathcal{N}_{K,t})^2.$$ By we calculate the expected value: $$\mathbb{E}(\mathcal{N}_{K,t}) =\int_{-\pi}^{\pi}\mathcal{N}_{K,t} \frac{d\theta}{2\pi}=\frac{\mathcal{N}_t}{K}.$$ Now we estimate $\mathbb{E}(\mathcal{N}_{K,t}^2)$. Since
$$\mathcal{N}_{K,t}(\theta)=\sum_{\substack{\hat z\in\mathcal{P}\\|\lambda(\hat z)|<t}}\mathds{1}_{[-\frac{\pi}{K}, \frac{\pi}{K}]}(\mathrm{Arg}(\lambda(\hat z))-\theta),$$
it follows that 
$$\mathcal{N}_{K,t}^2(\theta) =\sum_{\substack{\hat z\in\mathcal{P}\\|\lambda(\hat z)|<t}}\sum_{\substack{\hat w\in\mathcal{P}\\|\lambda(\hat w)|<t}}\mathds{1}_{[-\frac{\pi}{K}, \frac{\pi}{K}]}(\mathrm{Arg}(\lambda(\hat z))-\theta)\mathds{1}_{[-\frac{\pi}{K}, \frac{\pi}{K}]}(\mathrm{Arg}(\lambda(\hat w))-\theta).$$
Hence 
$$\mathbb{E}(\mathcal{N}_{K,t}^2)=\sum_{\substack{\hat z\in\mathcal{P}\\|\lambda(\hat z)|<t}}\sum_{\substack{\hat w\in\mathcal{P}\\|\lambda(\hat w)|<t}}\mathbb{E}(\mathds{1}_{[-\frac{\pi}{K}, \frac{\pi}{K}]}(\mathrm{Arg}(\lambda(\hat z))-\theta)\mathds{1}_{[-\frac{\pi}{K}, \frac{\pi}{K}]}(\mathrm{Arg}(\lambda(\hat w))-\theta)).$$
If $\hat z=\hat w$, we have that 
$$\mathbb{E}(\mathds{1}_{[-\frac{\pi}{K}, \frac{\pi}{K}]}(\mathrm{Arg}(\lambda(\hat z))-\theta)\mathds{1}_{[-\frac{\pi}{K}, \frac{\pi}{K}]}(\mathrm{Arg}(\lambda(\hat w))-\theta))=\mathbb{E}(\mathds{1}_{[-\frac{\pi}{K}, \frac{\pi}{K}]}(\mathrm{Arg}(\lambda(\hat z))-\theta)^2)=\frac{1}{K}.$$
It follows that 
$$\sum_{\substack{\hat z\in\mathcal{P}\\|\lambda(\hat z)|<t}}\sum_{\hat w=\hat z}\mathbb{E}(\mathds{1}_{[-\frac{\pi}{K}, \frac{\pi}{K}]}(\mathrm{Arg}(\lambda(\hat z))-\theta)\mathds{1}_{[-\frac{\pi}{K}, \frac{\pi}{K}]}(\mathrm{Arg}(\lambda(\hat w))-\theta))=\sum_{\substack{\hat z\in\mathcal{P}\\|\lambda(\hat z)|<t}}\frac{1}{K}=\frac{\mathcal{N}_t}{K}.$$

If $\hat z\not=\hat w$, we have that $$\mathds{1}_{[-\frac{\pi}{K}, \frac{\pi}{K}]}(\mathrm{Arg}(\lambda(\hat z))-\theta)\mathds{1}_{[-\frac{\pi}{K}, \frac{\pi}{K}]}(\mathrm{Arg}(\lambda(\hat w))-\theta)\not=0$$ for some $\theta$ if and only if $|\mathrm{Arg}(\lambda(\hat z))-\mathrm{Arg}(\lambda(\hat w)))|< 2\pi/K$. 
In this case we have 
$$\mathbb{E}(\mathds{1}_{[-\frac{\pi}{K}, \frac{\pi}{K}]}(\mathrm{Arg}(\lambda(\hat z))-\theta)\mathds{1}_{[-\frac{\pi}{K}, \frac{\pi}{K}]}(\mathrm{Arg}(\lambda(\hat w))-\theta))\le\frac{1}{K}.$$
It follows that 
\begin{align*}
&\sum_{\substack{\hat z\in\mathcal{P}\\|\lambda(\hat z)|<t}}\sum_{\substack{\hat w\in\mathcal{P}\\|\lambda(\hat w)|<t\\\hat w\not=\hat z}}\mathbb{E}(\mathds{1}_{[-\frac{\pi}{K}, \frac{\pi}{K}]}(\mathrm{Arg}(\lambda(\hat z))-\theta)\mathds{1}_{[-\frac{\pi}{K}, \frac{\pi}{K}]}(\mathrm{Arg}(\lambda(\hat w))-\theta))\\
&=\sum_{\substack{\hat z\in\mathcal{P}\\|\lambda(\hat z)|<t}}\sum_{\substack{\hat w\in\mathcal{P}\\|\lambda(\hat w)|<t\\\hat w\not=\hat z\\|\mathrm{Arg}(\lambda(\hat z))-\mathrm{Arg}(\lambda(\hat w)))|< 2\pi/K}}\mathbb{E}(\mathds{1}_{[-\frac{\pi}{K}, \frac{\pi}{K}]}(\mathrm{Arg}(\lambda(\hat z))-\theta)\mathds{1}_{[-\frac{\pi}{K}, \frac{\pi}{K}]}(\mathrm{Arg}(\lambda(\hat w))-\theta))\\
&\le\sum_{\substack{\hat z\in\mathcal{P}\\|\lambda(\hat z)|<t}}\sum_{\substack{\hat w\in\mathcal{P}\\|\lambda(\hat w)|<t\\\hat w\not=\hat z\\|\mathrm{Arg}(\lambda(\hat z))-\mathrm{Arg}(\lambda(\hat w)))|< 2\pi/K}}\frac{1}{K}\\
&\sim\mathcal{N}_t\frac{\mathcal{N}_t}{K}\frac{1}{K},
%&=\left(\frac{\mathrm{Li}(t^\delta)}{K}\right)^2.
\end{align*}
where the last asymptotic estimate is from \eqref{equ:introI}. Since $\mathcal{N}_t\ll K$, we have that 
$$\mathbb{E}(\mathcal{N}_{K,t}^2)\sim\frac{\mathcal{N}_t}{K}+\left(\frac{\mathcal{N}_t}{K}\right)^2.$$
Therefore,
$$\mathrm{Var}(\mathcal{N}_{K,t})=\mathbb{E}(\mathcal{N}_{K,t}^2)-\mathbb{E}(\mathcal{N}_{K,t})^2\sim \frac{\mathcal{N}_t}{K}.$$ 
\end{proof}

\bibliographystyle{siam}
\bibliography{references}
\end{document}